\documentclass[reqno,12pt]{amsart}
\usepackage{setspace,tikz,xcolor,mathrsfs,listings,multicol}
\usepackage{rotating}
\usepackage[vcentermath]{youngtab}
\usepackage{fullpage}
\usepackage{enumerate}
\usepackage{booktabs}
\usepackage{cite}
\usepackage[all,cmtip]{xy}
\usetikzlibrary{arrows,matrix}
\tikzset{tab/.style={matrix of math nodes,column sep=-.35, row sep=-.35,text height=7pt,text width=7pt,align=center,inner sep=2,font=\footnotesize}}

\usepackage{setspace}

\usepackage[colorlinks=true, pdfstartview=FitV, linkcolor=blue, citecolor=blue, urlcolor=blue]{hyperref}

\newcommand{\g}{\mathfrak{g}}

\newcommand{\clfw}{\overline{\omega}} 
\newcommand{\fwA}{\eta} 
\newcommand{\inner}[2]{\left\langle #1, #2 \right\rangle}
\newcommand{\iso}{\cong}

\newcommand{\sgn}{\operatorname{sgn}}

\DeclareMathOperator{\wt}{wt} 
\DeclareMathOperator{\lev}{lev} 
\DeclareMathOperator{\stab}{stab} 

\newcommand{\mcC}{\mathcal{C}}

\newcommand{\ZZ}{\mathbb{Z}}

\newcommand{\bon}{\overline{1}}
\newcommand{\btw}{\overline{2}}
\newcommand{\bth}{\overline{3}}
\newcommand{\bfo}{\overline{4}}
\newcommand{\bfive}{\overline{5}}
\newcommand{\bsix}{\overline{6}}
\newcommand{\bseven}{\overline{7}}

\definecolor{darkred}{rgb}{0.7,0,0} 
\newcommand{\defn}[1]{{\color{darkred}\emph{#1}}} 

\definecolor{uqgold}{RGB}{196, 158, 54} 
\definecolor{uqpurple}{RGB}{73, 7, 94} 

\usepackage{listings}
\lstdefinelanguage{Sage}[]{Python}
{morekeywords={False,sage,True},sensitive=true}
\definecolor{dblackcolor}{rgb}{0.0,0.0,0.0}
\definecolor{dbluecolor}{rgb}{0.01,0.02,0.7}
\definecolor{dgreencolor}{rgb}{0.2,0.4,0.0}
\definecolor{dgraycolor}{rgb}{0.30,0.3,0.30}

\theoremstyle{plain}
\newtheorem{thm}{Theorem}[section]
\newtheorem{lemma}[thm]{Lemma}
\newtheorem{conj}[thm]{Conjecture}
\newtheorem{prop}[thm]{Proposition}
\newtheorem{cor}[thm]{Corollary}
\theoremstyle{definition}

\newtheorem{remark}[thm]{Remark}
\numberwithin{equation}{section}

\usepackage[colorinlistoftodos]{todonotes}

\begin{document}
\title[KR crystals $B^{7,s}$ for type $E_7^{(1)}$]{Kirillov--Reshetikhin crystals $B^{7,s}$ for type $E_7^{(1)}$}

\author[R.~Biswal]{Rekha Biswal}
\address[R. Biswal]{Max-Planck-Institut f\"ur Mathematik, Vivatsgasse 7, 53111 Bonn, Germany}
\curraddr{University of Edinburgh, Old College, South Bridge, Edinburgh EH8 9YL, United Kingdom}
\email{rbiswal@ed.ac.uk}
\urladdr{https://rekhabiswal.github.io/}

\author[T.~Scrimshaw]{Travis Scrimshaw}
\address[T. Scrimshaw]{School of Mathematics and Physics, The University of Queensland, St.\ Lucia, QLD 4072, Australia}
\email{tcscrims@gmail.com}
\urladdr{https://people.smp.uq.edu.au/TravisScrimshaw/}


\thanks{R.B.\ was partially supported by the NSERC discovery grant of her Postdoc supervisor Michael Lau at Universit\'e Laval.
T.S.\ was partially supported by the Australian Research Council grant DP170102648.}

\begin{abstract}
We construct a combinatorial crystal structure on the Kirillov--Reshetikhin crystal $B^{7,s}$ in type $E_7^{(1)}$, where $7$ is the unique node in the orbit of $0$ in the affine Dynkin diagram.
We then describe the combinatorial $R$-matrix $R \colon B^{7,s} \otimes B^{7,s'} \to B^{7,s'} \otimes B^{7,s}$.
\end{abstract}

\maketitle

\section{Introduction}
\label{sec:introduction}

An important class of finite-dimensional representations for affine Lie algebras are the \defn{Kirillov--Reshetikhin (KR) modules}, which are characterized by their Drinfel'd polynomials~\cite{CP95,CP98}. We denote a KR module by $W^{r,s}$, where $r$ is a node of the classical Dynkin diagram and $s$ is a positive integer. KR modules have been well-studied and have many interesting properties. For example, their characters (resp.~$q$-characters) are solutions of the Q-system (resp.~T-system)~\cite{Hernandez10} (see also~\cite{KNS11} and references therein). Moreover, graded (Demazure-type) characters of tensor products of single-column KR modules are (nonsymmetric) Macdonald polynomials at $t=0$ for untwisted affine types~\cite{LNSSS14,LNSSS14II,LNSSS15}.

One significant aspect of a KR module $W^{r,s}$ is that it (conjecturally) admits a crystal base~\cite{HKOTY99,HKOTT02} despite not being a highest weight module. The corresponding crystal of $W^{r,s}$ is called a \defn{Kirillov--Reshetikhin (KR) crystal} and denoted by $B^{r,s}$. KR crystals have been shown to exist in all nonexceptional types in~\cite{OS08}, types $G_2^{(1)}$ and $D_4^{(3)}$ in~\cite{Naoi18}, for a number of nodes in exceptional types~\cite{BS20,NS19}, and for $r$ being in the orbit of or adjacent to $0$ in all affine types from the general theory~\cite{KKMMNN91,KKMMNN92}. An open problem is to determine a uniform model for KR crystals. This has been achieved for $B^{r,1}$ by using Kashiwara's construction of projecting an extremal level-zero module/crystal~\cite{K02}. This was done explicitly by Naito and Sagaki using Lakshmibai--Seshadri (LS) paths~\cite{NS03,NS06,NS08II}. The construction of Kashiwara was also shown to partially extend to general $B^{r,s}$ in nonexceptional affine types (conjecturally in all affine types)~\cite{LS18}.
In contrast, the models in~\cite{FOS09,JS10,KMOY07,Yamane98} are all type-dependent, but are given for $B^{r,s}$ for all $s$.

KR crystals are connected with mathematical physics. For instance, tensor products of KR modules are used to describe certain vertex models and are related with Heisenberg spin chains by the $X=M$ conjecture of~\cite{HKOTY99,HKOTT02}.
The $X=M$ conjecture implies a fermionic formula for the graded characters of a tensor product of KR crystals; see~\cite{OSS18,Scrimshaw17} for recent progress. Furthermore, tensor products of KR crystals describe the dynamics of soliton cellular automata, a generalization of the Takehashi--Satsuma box-ball system (which is an ultradiscrete version of the Korteweg--de Vries (KdV) equation). We refer the reader to~\cite{IKT12,LS17} for more details.
Another important (conjectural) property of KR crystals is that they are perfect~\cite{FOS10,KKMMNN91,KKMMNN92,KMOY07,Yamane98}, a technical condition that allows highest weight crystals to be modeled using a semi-infinite tensor product known as the Kyoto path model~\cite{KKMMNN92}.

In this note, we give a combinatorial model for the KR crystal $B^{7,s}$ in type $E_7^{(1)}$, where $7$ is the unique node in the orbit of $0$ in the Dynkin diagram (see Figure~\ref{fig:Dynkin} below). We achieve this by considering the (Levi) decomposition of the classical (type $E_7$) highest weight crystal $B(s\clfw_7)$ into $A_6$ highest weight crystals, which is multiplicity free. From this, we reconstruct the $A_7$ decomposition of $B^{7,s}$ since $B^{7,s} \iso B(s\clfw_7)$ as $E_7$ crystals and the decomposition of $A_7$ highest weight crystals into $A_6$ crystals is multiplicity free. We note that the KR crystal $B^{7,s}$ exists since $7$ is a minuscule node, so $B^{7,s}$ is irreducible as a classical crystal.
The novelty of our approach is doing a further Levi decomposition and reconstructing the affine action to a type $A_7$ crystal rather than through the classical decomposition.
We then given an explicit description of the combinatorial $R$-matrix $R \colon B^{7,s} \otimes B^{7,s'} \to B^{7,s'} \otimes B^{7,s}$.
We note that the local energy function is given by~\cite[Thm.~7.5]{Scrimshaw17}.

As a potential application of our results, the combinatorial $R$-matrix allows us to study soliton cellular automata of $B^{7,1}$ using different techniques from~\cite{LS18}. Moreover, our results could potentially be used to show that $B^{7,s}$ is a perfect crystal of level~$s$.

This paper is organized as follows.
In Section~\ref{sec:background}, we give the necessary background.
In Section~\ref{sec:results}, we give our main results.
In Section~\ref{sec:conjectures}, we give a conjecture about the decomposition of $B^{1,s}$ into $A_7$ crystals in an effort to prove~\cite[Conj.~3.26]{JS10}.

\subsection*{Acknowledgments}

The authors thank the referee for useful comments on our manuscript.

\section{Background}
\label{sec:background}

Let $\g$ be an affine Kac--Moody Lie algebra with index set $I$, Cartan matrix $(A_{ij})_{i,j \in I}$, simple roots $(\alpha_i)_{i \in I}$, fundamental weights $(\omega_i)_{i \in I}$, and simple coroots $(\alpha_i^{\vee})_{i \in I}$.
Let $U_q(\g)$ denote the corresponding (Drinfel'd--Jimbo) quantum group, and we will be using $U_q'(\g) := U_q([\g, \g])$, which has weight lattice $P = \sum_{i \in I} \ZZ\omega_i$.
Let $P^+$ denote the positive weight lattice.
Let $Q$ be the root lattice with $Q^+$ being the positive root lattice.
We denote the canonical pairing $\langle\ ,\ \rangle \colon P^{\vee} \times P \to \ZZ$, which is given by $\inner{\alpha_i^{\vee}}{\alpha_j} = A_{ij}$.

Recall that $\lev(\lambda) := \langle c, \lambda \rangle$ is the level of the weight $\lambda$, where $c$ is the canonical central element of $\g$. In particular, for $\g$ of type $E_7^{(1)}$, we have
\begin{equation}
\label{eq:levels}
\begin{aligned}
\lev(\omega_0) & = 1, & \qquad
\lev(\omega_1) & = 2, & \qquad
\lev(\omega_2) & = 2, & \qquad
\lev(\omega_3) & = 3, \\
\lev(\omega_4) & = 4, &
\lev(\omega_5) & = 3, &
\lev(\omega_6) & = 2, &
\lev(\omega_7) & = 1.
\end{aligned}
\end{equation}
We denote the dominant weights of level $\ell$ by $P_\ell^+$.

Let $\g_0$ denote the canonical simple Lie algebra given by the index set $I_0 = I \setminus \{0\}$, and $U_q(\g_0)$ the corresponding quantum group.
Let $P_0$ and $Q_0$ be the weight and root lattice of $\g_0$, and let $\clfw_i$ be the natural projection of the fundamental weight $\omega_i$ onto $P_0$.
Let $W_0$ be the Weyl group of $\g_0$.

\begin{figure}
\[
\begin{tikzpicture}[scale=0.65, baseline=0]
\draw (0 cm,0) -- (12 cm,0);
\draw (6 cm, 0 cm) -- +(0,2 cm);
\draw[fill=white] (0 cm, 0 cm) circle (.25cm) node[below=5pt]{$0$};
\draw[fill=white] (2 cm, 0 cm) circle (.25cm) node[below=5pt]{$1$};
\draw[fill=white] (4 cm, 0 cm) circle (.25cm) node[below=5pt]{$3$};
\draw[fill=white] (6 cm, 0 cm) circle (.25cm) node[below=5pt]{$4$};
\draw[fill=white] (8 cm, 0 cm) circle (.25cm) node[below=5pt]{$5$};
\draw[fill=white] (10 cm, 0 cm) circle (.25cm) node[below=5pt]{$6$};
\draw[fill=white] (12 cm, 0 cm) circle (.25cm) node[below=5pt]{$7$};
\draw[fill=white] (6 cm, 2 cm) circle (.25cm) node[right=4pt]{$2$};
\end{tikzpicture}
\]
\caption{The Dynkin diagram of type $E_7^{(1)}$.}
\label{fig:Dynkin}
\end{figure}
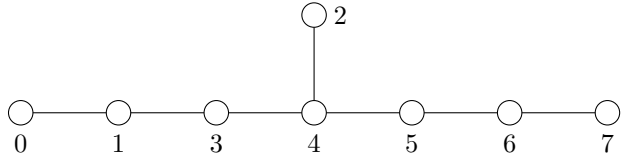

\subsection{Crystals}

An \defn{abstract $U_q(\g)$-crystal} is a set $B$ endowed with \defn{crystal operators} $e_i, f_i \colon B \to B \sqcup \{0\}$, for $i \in I$, and \defn{weight function} $\wt \colon B \to P$ that satisfy the following conditions:
\begin{itemize}
\item[(1)] $\varphi_i(b) = \varepsilon_i(b) + \inner{\alpha_i^{\vee}}{\wt(b)}$, for all $b \in B$ and $i \in I$,
\item[(2)] $f_i b = b'$ if and only if $b = e_i b'$, for $b, b' \in B$ and $i \in I$,
\item[(3)] $\wt(f_i b) = \wt(b) - \alpha_i$ if $f_i b \neq 0$;
\end{itemize}
where the statistics $\varepsilon_i, \varphi_i \colon  B \to \ZZ_{\geq 0}$ are defined by
\[
\varepsilon_i(b) := \max \{ k \mid e_i^k b \neq 0 \}\,,
\qquad \qquad \varphi_i(b) := \max \{ k \mid f_i^k b \neq 0 \}\,.
\]

\begin{remark}
The definition of an abstract crystal given in this paper is sometimes called a \defn{regular} or \defn{seminormal} abstract crystal in the literature. See, \textit{e.g.},~\cite{BS17} for the more general definition.
\end{remark}

Using the axioms, we identify $B$ with an $I$-edge colored weighted directed graph whose vertices are $B$ and having an $i$-colored edge $b \xrightarrow[\quad]{i} b'$ if and only if $f_i b = b'$.
Therefore, we can depict an entire $i$-string through an element $b \in B$ diagrammatically by
\[
e_i^{\varepsilon_i(b)}b \xrightarrow[\hspace{15pt}]{i}
\cdots \xrightarrow[\hspace{15pt}]{i}
e_i^2 b \xrightarrow[\hspace{15pt}]{i}
e_i b \xrightarrow[\hspace{15pt}]{i}
b \xrightarrow[\hspace{15pt}]{i}
f_i b \xrightarrow[\hspace{15pt}]{i}
f_i^2 b \xrightarrow[\hspace{15pt}]{i}
\cdots \xrightarrow[\hspace{15pt}]{i}
f_i^{\varphi_i(b)} b.
\]
Let $J \subseteq I$.
An element $b \in B$ is \defn{$J$-highest (resp.\ lowest) weight} if $e_i b = 0$ (resp.~$f_i b = 0$) for all $i \in J$.
When $J = I$, we simply say $b$ is \defn{highest (resp.\ lowest) weight}.


For abstract $U_q(\g)$-crystals $B_1, B_2, \dotsc, B_L$, the action of the crystal operators on the \defn{tensor product} $B_L \otimes \cdots \otimes B_2 \otimes B_1$, which equals the Cartesian product $B_L \times \cdots \times B_2 \times B_1$ as sets, can be defined by the \defn{signature rule}.
Let $b := b_L \otimes \cdots \otimes b_2 \otimes b_1 \in B$, and for $i \in I$, we write
\[
\underbrace{-\cdots-}_{\varphi_i(b_L)}\ 
\underbrace{+\cdots+}_{\varepsilon_i(b_L)}\ 
\cdots\ 
\underbrace{-\cdots-}_{\varphi_i(b_1)}\ 
\underbrace{+\cdots+}_{\varepsilon_i(b_1)}.
\]
Then by successively deleting consecutive $+-$-pairs (in that order), we obtain a sequence
\[
\sgn_i(b) :=
\underbrace{-\cdots-}_{\varphi_i(b)}\
\underbrace{+\cdots+}_{\varepsilon_i(b)}\,,
\]
called the \defn{reduced signature}.
If there does not exist a $+$ (resp.~$-$) in $\sgn_i(b)$, then $e_i b = 0$ (resp.~$f_i b = 0)$.
Otherwise, suppose $1 \leq j_-\leq j_+ \leq L$ are such that $b_{j_-}$ contributes the rightmost $-$ in $\sgn_i(b)$ and $b_{j_+}$ contributes the leftmost $+$ in $\sgn_i(b)$. 
Then, we have
\begin{align*}
e_i b &:= b_L \otimes \cdots \otimes b_{j_++1} \otimes e_ib_{j_+} \otimes b_{j_+-1} \otimes \cdots \otimes b_1\,, \\
f_i b &:= b_L \otimes \cdots \otimes b_{j_-+1} \otimes f_ib_{j_-} \otimes b_{j_--1} \otimes \cdots \otimes b_1\,.
\end{align*}

\begin{remark}
Our tensor product convention follows~\cite{BS17}, which is opposite to that of Kashiwara~\cite{K91}.
\end{remark}

Let $B_1$ and $B_2$ be two abstract $U_q(\g)$-crystals.
A \defn{crystal morphism} $\psi \colon B_1 \to B_2$ is a map $B_1 \sqcup \{0\} \to B_2 \sqcup \{0\}$ with $\psi(0) = 0$, such that the following properties hold for all $b \in B_1$ and $i \in I$:
\begin{itemize}
\item[(1)] if $\psi(b) \in B_2$, then $\wt\bigl(\psi(b)\bigr) = \wt(b)$, $\varepsilon_i\bigl(\psi(b)\bigr) = \varepsilon_i(b)$, and $\varphi_i\bigl(\psi(b)\bigr) = \varphi_i(b)\,;$
\item[(2)] we have $\psi(e_i b) = e_i \psi(b)$ if $\psi(e_i b) \neq 0$ and $e_i \psi(b) \neq 0\,;$
\item[(3)] we have $\psi(f_i b) = f_i \psi(b)$ if $\psi(f_i b) \neq 0$ and $f_i \psi(b) \neq 0\,.$
\end{itemize}
An \defn{embedding} (resp.~\defn{isomorphism}) is a crystal morphism such that the induced map $B_1 \sqcup \{0\} \to B_2 \sqcup \{0\}$ is an embedding (resp.~bijection).

An abstract crystal $B$ is a \defn{$U_q(\g)$-crystal} if $B$ is the crystal basis of some $U_q(\g)$-module.
Kashiwara~\cite{K91} has shown that the irreducible highest weight module $V(\lambda)$, for $\lambda \in P^+$, admits a crystal basis denoted $B(\lambda)$.
The highest weight crystal $B(\lambda)$ is generated by a unique highest weight element $u_{\lambda}$ that satisfies $\wt(u_{\lambda}) = \lambda$.

\subsection{Minuscule crystals}

We say a highest weight $U_q(\g_0)$-crystal $B(\lambda)$ is \defn{minuscule} if $W_0$ acts transitively on $B(\lambda)$.
In other words, there exists a bijection between $B(\lambda)$ and $W_0^{\lambda}$, the set of minimal length coset representatives of $W_0 / \stab_{W_0}(\lambda)$ (recall $\stab_{W_0}(\lambda) := \{ w \in W_0 \mid w \lambda = \lambda\}$ is the stabilizer of $\lambda$ and a parabolic subgroup of $W_0$).
Indeed, consider a minimal length coset representative $w \in W_0^{\lambda}$ with a reduced expression $s_{i_1} s_{i_2} \cdots s_{i_k}$, then the corresponding element is $u_{w\lambda} := f_{i_1} f_{i_2} \cdots f_{i_k} u_{\lambda}$.
We note that the element $u_{w\lambda}$ is independent of the choice of reduced expression.

For a minuscule representation $B(\lambda)$, we can characterize the elements in $B(s\lambda)$ as follows.
Recall that since $B(\lambda)$ is a highest weight crystal, it can be considered as the Hasse diagram of a poset with $u_{\lambda}$ being the smallest element.

\begin{prop}[{\cite[Prop.~7.29]{Scrimshaw17}}]
\label{prop:minuscule}
Let $B(\lambda)$ be a minuscule representation.
The crystal $B(s\lambda)$ is isomorphic to the set of \defn{semistandard tableaux} whose shape is a single row of length~$s$
\[
T = 
\begin{array}{|c|c|c|c|c|}
\hline
x_1 & x_2 & \cdots & x_s
\\\hline
\end{array},
\]
with $x_1 \leq x_2 \leq \cdots \leq x_s$ in $B(\lambda)$,
and the crystal structure is given by considering $T$ as the element $x_1 \otimes x_2 \otimes \cdots \otimes x_s \in B(\lambda)^{\otimes s}$.
\end{prop}

\subsection{Type \texorpdfstring{$A_n$}{An} crystals}

In this section, we consider the Lie algebra of type $A_n$, which is $\mathfrak{sl}_{n+1}$.
We denote the fundamental weights of type $A_n$ by $\{\fwA_i \mid 1 \leq i \leq n\}$.
Recall that we have a natural bijection between $P^+$ and partitions of length at most $n$ by $\fwA_i$ corresponding to a column of height $i$.
Let $B(\fwA_1)$ denote the crystal of the vector representation of $\mathfrak{sl}_{n+1}$:
\[
\begin{tikzpicture}
\node (1) at (0,0) {$\boxed{1}$};
\node (2) at (2,0) {$\boxed{2}$};
\node (3) at (4,0) {$\boxed{3}$};
\node (d) at (6,0) {$\cdots$};
\node (n) at (8,0) {$\boxed{n}$};
\node (np1) at (10,0) {$\boxed{n+1}$};
\draw[->,red] (1) -- (2) node[midway,above] {\tiny $1$};
\draw[->,blue] (2) -- (3) node[midway,above] {\tiny $2$};
\draw[->] (3) -- (d) node[midway,above] {\tiny $3$};
\draw[->,uqgold] (d) -- (n) node[midway,above] {\tiny $n-1$};
\draw[->,uqpurple] (n) -- (np1) node[midway,above] {\tiny $n$};
\end{tikzpicture}.
\]
Furthermore, the crystal $B(\fwA)$ can be described by \defn{semistandard Young tableaux (SSYT)}, written in English convention, whose entries are at most $n+1$.
The crystal structure is given by embedding a SSYT $T \in B(\fwA)$ into $B(\fwA_1)^{\otimes \lvert \fwA \rvert}$ by the reverse Far-Eastern reading word: reading bottom-to-top and left-to-right.

Next, we recall the Levi branching rule $\mathfrak{sl}_{n+1} \searrow \mathfrak{sl}_n$ given at the level of crystals.
In terms of the crystal graph, we simply remove all $n$-colored edges.
For the SSYT, this amounts to fixing all $n+1$'s that appear.
Since any $n+1$ must be the bottom entry of every column and the largest entry in a given row, we obtain the following statement (which is well-known to experts).

\begin{prop}
\label{prop:A_branching}
As $U_q(\mathfrak{sl}_n)$-crystals, we have
\[
B(\fwA) \iso \bigoplus_{\mu} B(\mu),
\]
where the sum is taken over all $\mu$ such that $\fwA / \mu$ is a horizontal strip (\textit{i.e.}\ a skew partition that does not contain a vertical domino).
\end{prop}

Note that the decomposition of Proposition~\ref{prop:A_branching} is \emph{multiplicity-free}.

We recall that there exists a natural order $2$ diagram automorphism $\sigma$ on type $A_n$ crystals given by $i^{\sigma} = n + 1 - i$.
Indeed, we define an automorphism of the weight lattice, which we abuse notation and also denote by $\sigma \colon P \to P$, by $\fwA_i \mapsto \fwA_{i^{\sigma}} = \fwA_{n + 1 - i}$;
in particular, we note that $\fwA \mapsto -w_0 \fwA$.
Next, we define a map also denoted by $\sigma \colon B(\fwA) \mapsto B( -w_0 \fwA )$ given by
\[
\sigma(f_{i_1} \cdots f_{i_k} u_{\fwA}) = f_{i_1^{\sigma}} \cdots f_{i_k^{\sigma}} u_{-w_0 \fwA}.
\]
We recall from~\cite{SS2006} that $\sigma = \vee \circ \ast$, where $\vee$ denotes the contragredient dual map and $\ast$ is the the Lusztig involution (which equals the Sch\"utzenberger involution~\cite{Lenart00}).

In the sequel, we require the $A_7$ Levi subalgebra $\g_2$ of the type $E_7^{(1)}$ affine Lie algebra $\g$ given by the index set $I_2 := I \setminus \{2\}$. Therefore, the fundamental weights correspond by
\[
\fwA_1 = \omega_7, \qquad
\fwA_2 = \omega_6, \qquad
\fwA_3 = \omega_5, \qquad
\fwA_4 = \omega_4, \qquad
\fwA_5 = \omega_3, \qquad
\fwA_6 = \omega_1, \qquad
\fwA_7 = \omega_0.
\]
We let $I_{0,2} := J \setminus \{0,2\}$ be index set of the the $A_6$ Levi subalgebra of $\g_0$ (of type $E_7$).

%
%
%

\section{Results}
\label{sec:results}

In this section, we give our main results.

We follow~\cite[Fig.~3]{JS10} and label an element $b \in B(\clfw_7)$ as a word in $I_0$ and $\overline{I}_0 := \{\bon, \dotsc, \bseven\}$, where for $i \in b$ such that $i \in I_0$ (resp.~$i \in \overline{I}_0$), we have $\varphi_i(b) = 1$ (resp.~$\varepsilon_i(b) = 1$).
See Figure~\ref{fig:basic_crystal} for the crystal $B(\clfw_7)$ with this labeling convention.

\begin{figure}
\[
\begin{tikzpicture}[>=latex,line join=bevel,scale=1.5,every node/.style={scale=0.8}]
\node (7) at (0,0) {$7$};
\node (7b6) at (1,0) {$\bseven6$};
\node (6b5) at (2,0) {$\bsix5$};
\node (5b4) at (3,0) {$\bfive4$};
\node (4b23) at (4,0) {$\bfo23$};
\node (2b3) at (5,0) {$\btw3$};
\node (3b12) at (4,-1) {$\bth12$};
\node (23b14) at (5,-1) {$\btw\bth14$};
\node (1b2) at (4,-2) {$\bon2$};
\node (12b4) at (5,-2) {$\bon\btw4$};
\node (4b15) at (6,-1) {$\bfo15$};
\node (14b35) at (6,-2) {$\bon\bfo35$};
\node (5b16) at (7,-1) {$\bfive16$};
\node (6b17) at (8,-1) {$\bsix17$};
\node (7b1) at (9,-1) {$\bseven1$};
\node (3b5) at (6,-3) {$\bth5$};
\node (15b36) at (7,-2) {$\bon\bfive36$};
\node (16b37) at (8,-2) {$\bon\bsix37$};
\node (17b3) at (9,-2) {$\bon\bseven3$};
\node (35b46) at (7,-3) {$\bth\bfive46$};
\node (36b47) at (8,-3) {$\bth\bsix47$};
\node (37b4) at (9,-3) {$\bth\bseven4$};
\node (4b26) at (7,-4) {$\bfo26$};
\node (46b257) at (8,-4) {$\bfo\bsix257$};
\node (47b25) at (9,-4) {$\bfo\bseven25$};
\node (2b6) at (7,-5) {$\btw6$};
\node (26b57) at (8,-5) {$\btw\bsix57$};
\node (27b5) at (9,-5) {$\btw\bseven5$};
\node (5b27) at (0,-4) {$\bfive27$};
\node (57b26) at (1,-4) {$\bfive\bseven26$};
\node (25b47) at (0,-5) {$\btw\bfive47$};
\node (257b46) at (1,-5) {$\btw\bfive\bseven46$};
\node (6b2) at (2,-4) {$\bsix2$};
\node (26b4) at (2,-5) {$\btw\bsix4$};
\node (4b37) at (0,-6) {$\bfo37$};
\node (47b36) at (1,-6) {$\bfo\bseven36$};
\node (46b35) at (2,-6) {$\bfo\bsix35$};
\node (5b3) at (3,-6) {$\bfive3$};
\node (3b17) at (0,-7) {$\bth17$};
\node (37b16) at (1,-7) {$\bth\bseven16$};
\node (36b15) at (2,-7) {$\bth\bsix15$};
\node (35b14) at (3,-7) {$\bth\bfive14$};
\node (4b12) at (4,-7) {$\bfo12$};
\node (2b1) at (5,-7) {$\btw1$};
\node (1b7) at (0,-8) {$\bon7$};
\node (17b6) at (1,-8) {$\bon\bseven6$};
\node (16b5) at (2,-8) {$\bon\bsix5$};
\node (15b4) at (3,-8) {$\bon\bfive4$};
\node (14b23) at (4,-8) {$\bon\bfo23$};
\node (12b3) at (5,-8) {$\bon\btw3$};
\node (3b2) at (4,-9) {$\bth2$};
\node (23b4) at (5,-9) {$\btw\bth4$};
\node (4b5) at (6,-9) {$\bfo5$};
\node (5b6) at (7,-9) {$\bfive6$};
\node (6b7) at (8,-9) {$\bsix7$};
\node (7b) at (9,-9) {$\bseven$};
\draw[->,blue] (3b12) -- node[midway,left] {\tiny $1$} (1b2);
\draw[->,blue] (23b14) -- node[midway,left] {\tiny $1$} (12b4);
\draw[->,blue] (4b15) -- node[midway,left] {\tiny $1$} (14b35);
\draw[->,blue] (5b16) -- node[midway,left] {\tiny $1$} (15b36);
\draw[->,blue] (6b17) -- node[midway,left] {\tiny $1$} (16b37);
\draw[->,blue] (7b1) -- node[midway,left] {\tiny $1$} (17b3);
\draw[->,blue] (3b17) -- node[midway,left] {\tiny $1$} (1b7);
\draw[->,blue] (37b16) -- node[midway,left] {\tiny $1$} (17b6);
\draw[->,blue] (36b15) -- node[midway,left] {\tiny $1$} (16b5);
\draw[->,blue] (35b14) -- node[midway,left] {\tiny $1$} (15b4);
\draw[->,blue] (4b12) -- node[midway,left] {\tiny $1$} (14b23);
\draw[->,blue] (2b1) -- node[midway,left] {\tiny $1$} (12b3);
\draw[->,darkred] (4b23) -- node[midway,above] {\tiny $2$} (2b3);
\draw[->,darkred] (3b12) -- node[midway,above] {\tiny $2$} (23b14);
\draw[->,darkred] (1b2) -- node[midway,above] {\tiny $2$} (12b4);
\draw[->,darkred] (4b26) -- node[midway,left] {\tiny $2$} (2b6);
\draw[->,darkred] (46b257) -- node[midway,left] {\tiny $2$} (26b57);
\draw[->,darkred] (47b25) -- node[midway,left] {\tiny $2$} (27b5);
\draw[->,darkred] (5b27) -- node[midway,left] {\tiny $2$} (25b47);
\draw[->,darkred] (57b26) -- node[midway,left] {\tiny $2$} (257b46);
\draw[->,darkred] (6b2) -- node[midway,left] {\tiny $2$} (26b4);
\draw[->,darkred] (4b12) -- node[midway,above] {\tiny $2$} (2b1);
\draw[->,darkred] (14b23) -- node[midway,above] {\tiny $2$} (12b3);
\draw[->,darkred] (3b2) -- node[midway,above] {\tiny $2$} (23b4);
\draw[->,dgreencolor] (4b23) -- node[midway,left] {\tiny $3$} (3b12);
\draw[->,dgreencolor] (2b3) -- node[midway,left] {\tiny $3$} (23b14);
\draw[->,dgreencolor] (14b35) -- node[midway,left] {\tiny $3$} (3b5);
\draw[->,dgreencolor] (15b36) -- node[midway,left] {\tiny $3$} (35b46);
\draw[->,dgreencolor] (16b37) -- node[midway,left] {\tiny $3$} (36b47);
\draw[->,dgreencolor] (17b3) -- node[midway,left] {\tiny $3$} (37b4);
\draw[->,dgreencolor] (4b37) -- node[midway,left] {\tiny $3$} (3b17);
\draw[->,dgreencolor] (47b36) -- node[midway,left] {\tiny $3$} (37b16);
\draw[->,dgreencolor] (46b35) -- node[midway,left] {\tiny $3$} (36b15);
\draw[->,dgreencolor] (5b3) -- node[midway,left] {\tiny $3$} (35b14);
\draw[->,dgreencolor] (14b23) -- node[midway,left] {\tiny $3$} (3b2);
\draw[->,dgreencolor] (12b3) -- node[midway,left] {\tiny $3$} (23b4);
\draw[->,black] (5b4) -- node[midway,above] {\tiny $4$} (4b23);
\draw[->,black] (23b14) -- node[midway,above] {\tiny $4$} (4b15);
\draw[->,black] (12b4) -- node[midway,above] {\tiny $4$} (14b35);
\draw[->,black] (35b46) -- node[midway,left] {\tiny $4$} (4b26);
\draw[->,black] (36b47) -- node[midway,left] {\tiny $4$} (46b257);
\draw[->,black] (37b4) -- node[midway,left] {\tiny $4$} (47b25);
\draw[->,black] (25b47) -- node[midway,left] {\tiny $4$} (4b37);
\draw[->,black] (257b46) -- node[midway,left] {\tiny $4$} (47b36);
\draw[->,black] (26b4) -- node[midway,left] {\tiny $4$} (46b35);
\draw[->,black] (35b14) -- node[midway,above] {\tiny $4$} (4b12);
\draw[->,black] (15b4) -- node[midway,above] {\tiny $4$} (14b23);
\draw[->,black] (23b4) -- node[midway,above] {\tiny $4$} (4b5);
\draw[->,magenta] (6b5) -- node[midway,above] {\tiny $5$} (5b4);
\draw[->,magenta] (4b15) -- node[midway,above] {\tiny $5$} (5b16);
\draw[->,magenta] (14b35) -- node[midway,above] {\tiny $5$} (15b36);
\draw[->,magenta] (3b5) -- node[midway,above] {\tiny $5$} (35b46);
\draw[dashed,magenta] (46b257) -- ++(.5,-.7);
\draw[dashed,magenta] (47b25) -- ++(.5,-.7);
\draw[dashed,magenta] (26b57) -- ++(.5,-.7);
\draw[dashed,magenta] (27b5) -- ++(.5,-.7);
\draw[->,magenta] (5b27) ++ (-0.5,0.7) -- node[midway,left] {\tiny $5$} (5b27);
\draw[->,magenta] (57b26) ++ (-0.5,0.7) -- node[midway,left] {\tiny $5$} (57b26);
\draw[->,magenta] (25b47) ++ (-0.5,0.7) -- node[midway,left] {\tiny $5$} (25b47);
\draw[->,magenta] (257b46) ++ (-0.5,0.7) -- node[midway,left] {\tiny $5$} (257b46);
\draw[->,magenta] (46b35) -- node[midway,above] {\tiny $5$} (5b3);
\draw[->,magenta] (36b15) -- node[midway,above] {\tiny $5$} (35b14);
\draw[->,magenta] (16b5) -- node[midway,above] {\tiny $5$} (15b4);
\draw[->,magenta] (4b5) -- node[midway,above] {\tiny $5$} (5b6);
\draw[->,orange] (7b6) -- node[midway,above] {\tiny $6$} (6b5);
\draw[->,orange] (5b16) -- node[midway,above] {\tiny $6$} (6b17);
\draw[->,orange] (15b36) -- node[midway,above] {\tiny $6$} (16b37);
\draw[->,orange] (35b46) -- node[midway,above] {\tiny $6$} (36b47);
\draw[->,orange] (4b26) -- node[midway,above] {\tiny $6$} (46b257);
\draw[->,orange] (2b6) -- node[midway,above] {\tiny $6$} (26b57);
\draw[->,orange] (57b26) -- node[midway,above] {\tiny $6$} (6b2);
\draw[->,orange] (257b46) -- node[midway,above] {\tiny $6$} (26b4);
\draw[->,orange] (47b36) -- node[midway,above] {\tiny $6$} (46b35);
\draw[->,orange] (37b16) -- node[midway,above] {\tiny $6$} (36b15);
\draw[->,orange] (17b6) -- node[midway,above] {\tiny $6$} (16b5);
\draw[->,orange] (5b6) -- node[midway,above] {\tiny $6$} (6b7);
\draw[->,brown] (7) -- node[midway,above] {\tiny $7$} (7b6);
\draw[->,brown] (6b17) -- node[midway,above] {\tiny $7$} (7b1);
\draw[->,brown] (16b37) -- node[midway,above] {\tiny $7$} (17b3);
\draw[->,brown] (36b47) -- node[midway,above] {\tiny $7$} (37b4);
\draw[->,brown] (46b257) -- node[midway,above] {\tiny $7$} (47b25);
\draw[->,brown] (26b57) -- node[midway,above] {\tiny $7$} (27b5);
\draw[->,brown] (5b27) -- node[midway,above] {\tiny $7$} (57b26);
\draw[->,brown] (25b47) -- node[midway,above] {\tiny $7$} (257b46);
\draw[->,brown] (4b37) -- node[midway,above] {\tiny $7$} (47b36);
\draw[->,brown] (3b17) -- node[midway,above] {\tiny $7$} (37b16);
\draw[->,brown] (1b7) -- node[midway,above] {\tiny $7$} (17b6);
\draw[->,brown] (6b7) -- node[midway,above] {\tiny $7$} (7b);
\end{tikzpicture}
\]
\caption{The crystal $B(\clfw_7)$ in type $E_7$.}
\label{fig:basic_crystal}
\end{figure}

\subsection{Multiplicity freeness of \texorpdfstring{$B(s\clfw_7)$}{B(s omega7)}}

We first prove that when we decompose the $E_7$ crystals $B(s\clfw_7)$ into the Levi subalgebra of type $A_6$, we obtain a multiplicity free decomposition.

The $I_{0,2}$-highest weight elements of $B(s\clfw_7)$ are single row tableaux of size $k$ whose entries consist of
\[
x_1 = 7, \qquad
x_2 = \bsix5, \qquad
x_3 = \bfo23, \qquad
\begin{aligned}
x_4 & = \bon2, \\[2pt]
x_{4'} & = \btw3,
\end{aligned} \qquad
x_5 = \bon\btw4 \qquad
x_6 = \btw6, \qquad
x_7 = \btw1.
\]
This can be seen by a direct computation using Proposition~\ref{prop:minuscule}, the crystal graph $B(\clfw_7)$, and the signature rule.
Note that in the crystal graph of $B(\clfw_7)$, we have
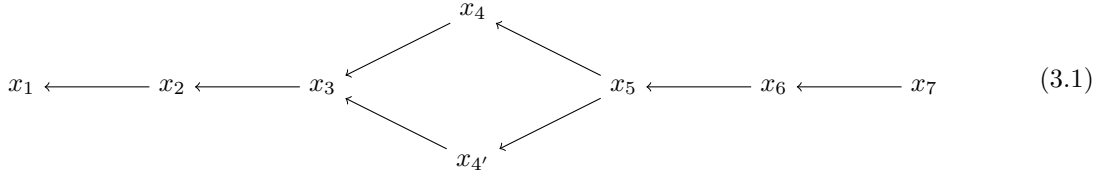
\begin{equation}
\label{eq:comp_graph_B7}
\begin{tikzpicture}[baseline=0]
\node (1) at (0,0) {$x_1$};
\node (2) at (2,0) {$x_2$};
\node (3) at (4,0) {$x_3$};
\node (4) at (6,1) {$x_4$};
\node (4p) at (6,-1) {$x_{4'}$};
\node (5) at (8,0) {$x_5$};
\node (6) at (10,0) {$x_6$};
\node (7) at (12,0) {$x_7$};
\draw[<-] (1) -- (2);
\draw[<-] (2) -- (3);
\draw[<-] (3) -- (4);
\draw[<-] (3) -- (4p);
\draw[<-] (4) -- (5);
\draw[<-] (4p) -- (5);
\draw[<-] (5) -- (6);
\draw[<-] (6) -- (7);
\end{tikzpicture}
\end{equation}
and so, nearly all of the elements $x_i$ are comparable except $x_4$ with $x_{4'}$.

We recall from~\cite[Def.~3.10]{JS10} that a \defn{(reduced) composition graph} $G_k(\clfw)$ is essentially the smallest acyclic digraph with loops whose vertices are elements of $B(\clfw)$ such that for any $s > 0$ and every $(I_0 \setminus \{k\})$-highest weight element $b_1 \otimes \cdots \otimes b_s \in B(s\clfw) \subseteq B(\clfw)^{\otimes s}$, the elements $(b_1, \dotsc, b_s)$ occurs as a subsequence of a directed path in $G_k(\clfw)$.
We remark by reversing the arrows and adding loops to every vertex in~\eqref{eq:comp_graph_B7}, we obtain the composition graph $G_1(\clfw_7)$.

Note that for a semistandard tableau $T \in B(s\clfw_7)$, we have
\[
\wt(T) = \sum_{i \in I} (a_i - a_{\overline{\imath}}) \clfw_i,
\]
where $a_i$ (resp.~$a_{\overline{\imath}}$) equals the number if $i$'s (resp.~$\overline{\imath}$'s) that appear in $T$. Therefore, from~\eqref{eq:comp_graph_B7} and the signature rule, we have the following.

\begin{lemma}
\label{lemma:J0_highest_weight}
Let $m_i$ denote the number of occurrences of $x_i$ in $T \in B(s\clfw_7)$. Then $T$ is a $I_{0,2}$-highest weight element if and only if entries in $T$ consist of $\{x_1, x_2, x_3, x_4, x_{4'}, x_5, x_6, x_7\}$ and
\[
m_2 \leq m_6, \qquad\qquad
m_3 \leq m_5, \qquad\qquad
m_4 + m_5 \leq m_7, \qquad\qquad
\min(m_4, m_{4'}) = 0,
\]
with $\sum_i m_i = s$.
Moreover, if $T \in B(s\clfw_7)$ is a $I_{0,2}$-highest weight element, then
\begin{equation}
\label{eq:wt_hw02elt}
\begin{aligned}
\wt(T) & = (m_7 - m_4 - m_5)\clfw_1 + (m_3 + m_4 - m_{4'} - m_5 - m_6 - m_7)\clfw_2 + (m_3+m_{4'})\clfw_3
\\ & \hspace{20pt} + (m_5 - m_3)\clfw_4 + m_2 \clfw_5 + (m_6 - m_2)\clfw_6 + m_1 \clfw_7.
\end{aligned}
\end{equation}
\end{lemma}

We note that the condition $\min(m_4, m_{4'}) = 0$ is precisely the fact that $x_4$ and $x_{4'}$ cannot simultaneously appear in an element of $B(s\clfw_7)$.

\begin{prop}
\label{prop:B1s_A6_decomp}
The decomposition of $B(s\clfw_7)$ into type $A_6$ crystals is given by
\[
B(s\clfw_7) \iso \bigoplus_{\mu} B(\mu),
\]
where
\[
\mu = (m_7 - m_4 - m_5)\fwA_6 + (m_3+m_{4'})\fwA_5 + (m_5 - m_3)\fwA_4 + m_2 \fwA_3 + (m_6 - m_2)\fwA_2 + m_1 \fwA_1
\]
such that $m_1, \dotsc, m_7$ satisfy
\begin{gather*}
m_2 \leq m_6, \qquad\qquad
m_3 \leq m_5, \qquad\qquad
m_4 + m_5 \leq m_7, \\
\min(m_4, m_{4'}) = 0, \qquad\qquad
s = m_1 + m_2 + m_3 + m_4 + m_{4'} + m_5 + m_6 + m_7.
\end{gather*}
Moreover, this decomposition is multiplicity free.
\end{prop}

\begin{proof}
The first claim follows immediately from Lemma~\ref{lemma:J0_highest_weight} and relabeling the fundamental weights.
For the second claim, consider a weight $\mu = \sum_{i=1}^6 a_i \fwA_i$ such that $B(\mu)$ appears in the $A_6$ decomposition of $B(s\clfw_7)$. Thus, we have
\begin{align*}
m_1 & = a_1,
&
m_2 & = a_3,
\\
m_6 & = a_2 + m_2 = a_2 + a_3,
&
m_3 & = a_5 - m_{4'},
\\
m_5 & = a_4 + m_3 = a_4 + a_5 - m_{4'},
&
m_7 & = a_6 + m_4 + m_5 = a_4 + a_5 + a_6 + m_4 - m_{4'}.
\end{align*}
Since $\min(m_4, m_{4'}) = 0$ with $m_4,m_{4'} \geq 0$, there exists a unique $m_4$ and $m_{4'}$ such that $m_4 - m_{4'} = C$ for any constant $C$.
Next, we have
\begin{align*}
s & = m_1 + m_2 + m_3 + m_4 + m_{4'} + m_5 + m_6 + m_7
\\ & = a_1 + a_3 + (a_5 - m_{4'}) + m_4 + m_{4'} + (a_4 + a_5 - m_{4'}) + (a_2 + a_3) + (a_4 + a_5 + a_6 + m_4 - m_{4'})
\\ & = a_1  + a_2 + 2a_3 + 2a_4 + 3a_5 + a_6 + 2(m_4 - m_{4'}).
\end{align*}
Hence, we have
\[
m_4 - m_{4'} = \frac{1}{2}(s - a_1 - a_2 - 2a_3 - 2a_4 - 3a_5 - a_6).
\]
Therefore, there is a unique $m_1, \dotsc, m_7$ that yields the weight $\mu$.
\end{proof}

\subsection{Reconstructing the \texorpdfstring{$A_7$}{A7} crystals}

In this section, we continue to use the notation of Proposition~\ref{prop:B1s_A6_decomp}.

\begin{lemma}
Let $T \in B^{7,s}$ be a $I_{0,2}$-highest weight element. Then
\begin{align*}
-\langle \alpha_0^{\vee}, \wt(T) \rangle =
m_1 + m_2 + m_3 + m_{4'}
\end{align*}
\end{lemma}

\begin{proof}
This follows from Equation~\eqref{eq:wt_hw02elt}, that all elements in $B^{7,s}$ are of level $0$, and~\eqref{eq:levels}.
\end{proof}

Since $\langle \alpha_0^{\vee}, \wt(T) \rangle \leq 0$ for all $I_{0,2}$-highest weight elements $T \in B^{7,s}$, we know that any potential $I_2$-highest weight element must have $m_1 = m_2 = m_3 = m_{4'} = 0$.
Hence, the possible $I_2$-dominant weights are
\[
\mu = (m_7 - m_4 - m_5)\fwA_6 + m_5 \fwA_4 + m_6 \fwA_2
\]
(\textit{i.e.}, $\langle \alpha_i^{\vee}, \mu \rangle \geq 0$ for all $i \in I_2$) such that
\[
m_4 + m_5 \leq m_7,\qquad\qquad
m_4 + m_5 + m_6 + m_7 = s.
\]
Because we can only remove horizontal strips for the branching rule from $A_7 \searrow A_6$ (Proposition~\ref{prop:A_branching}), each of such $I_2$-dominant weights $\mu$ must correspond to a $I_2$-highest weight component $B(\mu)$.
Hence, we obtain the following.

\begin{prop}
\label{prop:B1s_A7_decomp}
The decomposition of $B(s\clfw_7)$ into type $A_7$ crystals is given by
\[
B(s\clfw_7) \iso \bigoplus_{\mu} B(\mu),
\]
where
\[
\mu = (m_7 - m_4 - m_5) \fwA_6 + m_5 \fwA_4 + m_6 \fwA_2
\]
such that $m_4,m_5,m_6,m_7$ satisfy
\[
m_4 + m_5 \leq m_7,
\qquad\qquad
s = m_4 + m_5 + m_6 + m_7.
\]
\end{prop}

Moreover, Proposition~\ref{prop:A_branching} states that for any element $T$ expressed as an $A_6$ tableau, we add in a horizontal strip to $T$ with every entry an $8$ such that every column has even height to obtain the representation as an $A_7$ tableau.
In particular, for a $I_{0,2}$-highest weight element $b$, we have
\[
\varepsilon_0(b) = m_1 + m_2 + m_3 + m_{4'},
\qquad\qquad
\varphi_0(b) = 0.
\]

Now we combine this to form a combinatorial crystal structure on $B^{7,s}$ by extending the $E_7$ crystal structure on $B(s\clfw_7)$ as follows. Let $\overline{\psi} \colon B(s\clfw_7) \to \bigoplus_{\overline{\mu}} B(\overline{\mu})$ be the $I_{0,2}$-crystal (\textit{i.e.}, type $A_6$) isomorphism given by Proposition~\ref{prop:B1s_A6_decomp}. From Proposition~\ref{prop:B1s_A7_decomp} and Proposition~\ref{prop:A_branching}, we can uniquely extend the image of $\overline{\psi}$ to highest weight crystals of type $A_7$. Therefore, we define
\[
e_0 := \overline{\psi}^{-1} \circ e^A_7 \circ \overline{\psi},
\qquad\qquad\qquad
f_0 := \overline{\psi}^{-1} \circ f^A_7 \circ \overline{\psi},
\]
where $e_7^A$ and $f_7^A$ are the crystal operators from this extended type $A_7$ crystal. Let $\mathcal{B}^{7,s}$ denote the corresponding crystal.

In order to show this is the combinatorial structure of KR crystal, we need the following uniqueness theorem. The proof is similar to~\cite[Thm.~3.15]{JS10} with $K = I_{0,2}$ and using Proposition~\ref{prop:B1s_A6_decomp} instead of~\cite[Lemma~3.12]{JS10}.

\begin{thm}
Let $\mathcal{B}$ and $\mathcal{B}'$ be two affine type $E_7^{(1)}$ crystals such that there exists a $I_2$-crystal (\textit{i.e.}, type $A_7$) isomorphism and $I_0$-crystal (\textit{i.e.}, type $E_7$) isomorphism
\[
\Psi_{I_2} \colon \mathcal{B}|_{I_2} \to \mathcal{B}'|_{I_2} \iso \bigoplus_{\mu} B(\mu)
\qquad\qquad
\Psi_{I_0} \colon \mathcal{B}|_{I_0} \to \mathcal{B}'|_{I_0} \iso B(s\clfw_7),
\]
where the direct sum is over $\mu$ given in Proposition~\ref{prop:B1s_A7_decomp}.
Then, we have $\Psi_J(b) = \Psi_{I_0}(b)$ for all $b \in \mathcal{B}$. Moreover, there exists an $I$-crystal isomorphism $\Psi \colon \mathcal{B} \to \mathcal{B}'$.
\end{thm}

\begin{cor}
We have $\mathcal{B}^{7,s} \iso B^{7,s}$.
\end{cor}

\begin{proof}
We have that $B^{7,s} \iso B(s\clfw_7)$ as $I_0$-crystals since the corresponding KR module is irreducible as a $U_q(\g_0)$-module~\cite{Chari01}. Hence, the KR crystal $B^{7,s}$ exists by~\cite{KKMMNN91,KKMMNN92}. We can then decompose the crystal $B(s\clfw_7)$ into $I_{0,2}$-crystals according to Proposition~\ref{prop:B1s_A6_decomp}, and so the $I_2$-crystal decomposition of $B^{7,s}$ is given by Proposition~\ref{prop:B1s_A7_decomp}. Hence, we have isomorphisms between $\mathcal{B}^{7,s}$ and $B^{7,s}$ as $I_0$-crystals and $I_2$-crystals because these decompositions are multiplicity free.
\end{proof}

We have verified that $B^{7,s}$ is a perfect crystal of level $s$ for all $s \leq 4$ using the implementation of $B^{7,s}$ in \textsc{SageMath}~\cite{sage} done by the second author.

\subsection{Combinatorial isomorphism}
\label{sec:jdt}

Let $\langle b \rangle$ denote the $I_{0,2}$-subcrystal generated by an element $b$. Consider the subcrystals
\begin{align*}
\mcC_1 & = \langle x_1 \rangle,
&
\mcC_2 & = \langle x_6 \rangle,
&
\mcC_3 & = \langle x_2 \otimes x_6 \rangle,
\\
\mcC_4 & = \langle x_5 \otimes x_7 \rangle,
&
\mcC_5 & = \langle x_3 \otimes x_5 \otimes x_7 \rangle,
&
\mcC_6 & = \langle x_7 \rangle.
\end{align*}
It is clear that we have $I_{0,2}$-crystal isomorphisms $\psi_i \colon B(\fwA_i) \to \mcC_i$ for all $i$. (Note that we could have $\mcC_7 = \langle x_4 \otimes x_7 \rangle$, which would correspond to $m_4$.) Let $\mcC_{5'} = \langle x_{4'} \rangle$, and we also have a $I_{0,2}$-crystal isomorphism $\psi_{5'} \colon B(\fwA_5) \to \mcC_{5'}$.

Therefore, we can construct a bijection between SSYT for $B(\mu)$ with the $A_6$ component
\[
\mcC_6^{\otimes a_6} \otimes \mcC_{5'}^{\otimes m_{4'}} \otimes \mcC_5^{\otimes a_5 - m_{4'}} \otimes \mcC_4^{\otimes a_4} \otimes \mcC_3^{\otimes a_3} \otimes \mcC_2^{\otimes a_2} \otimes \mcC_1^{\otimes a_1}
\]
by applying $\psi_i$ on each column of height $i$, with possibly $\psi_{5'}$ on columns of height $5$, of a SSYT (of type $A_6$). Note that the result is ``out-of-order'' in that it does not result is $x_{i_1} \otimes \cdots \otimes x_{i_{\ell}}$ with $i_1 \leq \cdots \leq i_{\ell}$ (\textit{i.e.}, a semistandard tableaux given by Proposition~\ref{prop:minuscule}). Thus it would remain to ``sort'' these elements to obtain an honest component of $B(s\clfw_7)$. However, we will instead do the reverse process: we refine the isomorphisms $\psi_i^{-1}$ to apply them to elements of $B(\clfw_7)$, which we combine to apply to an element of $B(s\clfw_7)$.

Using the tensor product rule, we obtain the following defining crystal isomorphism with a tensor product of $A_6$ columns:
\[
x_1 \mapsto \young(1)\,,
\; 
x_2 \mapsto \young(3)\,,
\;
x_3 \mapsto \young(5)\,,
\;
x_{4} \mapsto \young(7)\,,
\;
x_{4'} \mapsto \young(1,2,3,4,5)\,,
\;
x_5 \otimes x_7 \mapsto \young(1,2,3,4)\,,
\;
x_6 \mapsto \young(1,2)\,,
\;
x_7 \mapsto \young(1,2,3,4,5,6)\,.
\]
We need one additional step to decouple $x_5 \otimes x_7$, which we do by
\[
x_5 \mapsto \young(1,2,3,4,7)\,.
\]
We could have arrived at this refined isomorphism by using the $A_6$ decomposition of $B(\clfw_7)$.
Now using our refined isomorphism and jeu-de-taquin, we obtain a crystal isomorphism $\Psi$ from elements of $B(s\clfw_7)$ to a direct sum of $A_6$ SSYT given by Proposition~\ref{prop:B1s_A6_decomp}. Recall that jeu-de-taquin is a crystal isomorphism.

Indeed, performing jeu-de-taquin (see, \textit{e.g.},~\cite{Fulton97}) on a highest weight element of the tensor product of single column SSYT, we obtain the SSYT
\[
\begin{tikzpicture}[scale=0.5]
\draw[fill=magenta!30] (0,0) rectangle (7,-1); 
\draw[fill=blue!30] (7,0) rectangle (10,-1); 
\draw[fill=uqpurple!30] (10,0) rectangle (16,-1); 
\draw[fill=darkred!40] (16,0) rectangle (21,-1); 
\draw[fill=uqgold!30] (21,0) rectangle (29,-1); 
\draw[fill=blue!30] (0,-1) rectangle (3,-2); 
\draw[fill=uqpurple!30] (3,-1) rectangle (9,-2); 
\draw[fill=darkred!40] (9,-1) rectangle (14,-2); 
\draw[fill=uqgold!30] (14,-1) rectangle (22,-2); 
\draw[fill=dgreencolor!30] (0,-2) rectangle (2,-3); 
\draw[fill=blue!30] (2,-2) rectangle (5,-3); 
\draw[fill=uqpurple!30] (5,-2) rectangle (11,-3); 
\draw[fill=uqgold!30] (11,-2) rectangle (19,-3); 
\draw[fill=blue!30] (0,-3) rectangle (3,-4); 
\draw[fill=uqpurple!30] (3,-3) rectangle (9,-4); 
\draw[fill=uqgold!30] (9,-3) rectangle (17,-4); 
\draw[fill=cyan!30] (0,-4) rectangle (6,-5); 
\draw[fill=blue!30] (6,-4) rectangle (9,-5); 
\draw[fill=uqgold!30] (9,-4) rectangle (17,-5); 
\draw[fill=uqgold!30] (0,-5) rectangle (8,-6); 
\draw[fill=black!20] (0,-6) rectangle (1,-7); 
\draw[fill=uqpurple!30] (1,-6) rectangle (7,-7); 
\node at (3.5,-0.5) {$m_1$};
\node at (8.5,-0.5) {$m_{4'}$};
\node at (13,-0.5) {$m_5$};
\node at (18.5,-0.5) {$m_6$};
\node at (25,-0.5) {$m_7$};
\node at (1.5,-1.5) {$m_{4'}$};
\node at (6,-1.5) {$m_5$};
\node at (11.5,-1.5) {$m_6$};
\node at (18,-1.5) {$m_7$};
\node at (1,-2.5) {$m_2$};
\node at (3.5,-2.5) {$m_{4'}$};
\node at (8,-2.5) {$m_5$};
\node at (15,-2.5) {$m_7$};
\node at (1.5,-3.5) {$m_{4'}$};
\node at (6,-3.5) {$m_5$};
\node at (13,-3.5) {$m_7$};
\node at (3,-4.5) {$m_3$};
\node at (7.5,-4.5) {$m_{4'}$};
\node at (13,-4.5) {$m_7$};
\node at (4,-5.5) {$m_7$};
\node at (0.5,-6.5) {$m_4$};
\node at (4,-6.5) {$m_5$};
\end{tikzpicture}
\]
where row $i$ is filled with only $i$. Note that we have included an entry from both $m_4$ and $m_{4'}$ for illustrative purposes, but both cannot appear. Augmenting the result with a horizontal strip of $8$'s to obtain all columns of even height (which agrees with Proposition~\ref{prop:B1s_A6_decomp} after ignoring the columns of height $8$), we have
\[
\begin{tikzpicture}[scale=0.5]
\draw[fill=magenta!30] (0,0) rectangle (7,-1); 
\draw[fill=blue!30] (7,0) rectangle (10,-1); 
\draw[fill=uqpurple!30] (10,0) rectangle (16,-1); 
\draw[fill=darkred!40] (16,0) rectangle (21,-1); 
\draw[fill=uqgold!30] (21,0) rectangle (29,-1); 
\draw[fill=blue!30] (0,-1) rectangle (3,-2); 
\draw[fill=uqpurple!30] (3,-1) rectangle (9,-2); 
\draw[fill=darkred!40] (9,-1) rectangle (14,-2); 
\draw[fill=uqgold!30] (14,-1) rectangle (22,-2); 
\draw (22,-1) rectangle (29,-2); 
\draw[fill=dgreencolor!30] (0,-2) rectangle (2,-3); 
\draw[fill=blue!30] (2,-2) rectangle (5,-3); 
\draw[fill=uqpurple!30] (5,-2) rectangle (11,-3); 
\draw[fill=uqgold!30] (11,-2) rectangle (19,-3); 
\draw[fill=blue!30] (0,-3) rectangle (3,-4); 
\draw[fill=uqpurple!30] (3,-3) rectangle (9,-4); 
\draw[fill=uqgold!30] (9,-3) rectangle (17,-4); 
\draw (17,-3) rectangle (19,-4); 
\draw[fill=cyan!30] (0,-4) rectangle (6,-5); 
\draw[fill=blue!30] (6,-4) rectangle (9,-5); 
\draw[fill=uqgold!30] (9,-4) rectangle (17,-5); 
\draw[fill=uqgold!30] (0,-5) rectangle (8,-6); 
\draw (8,-5) rectangle (17,-6); 
\draw[fill=black!20] (0,-6) rectangle (1,-7); 
\draw[fill=uqpurple!30] (1,-6) rectangle (7,-7); 
\draw (0,-7) rectangle (7,-8); 
\node at (3.5,-0.5) {$m_1$};
\node at (8.5,-0.5) {$m_{4'}$};
\node at (13,-0.5) {$m_5$};
\node at (18.5,-0.5) {$m_6$};
\node at (25,-0.5) {$m_7$};
\node at (1.5,-1.5) {$m_{4'}$};
\node at (6,-1.5) {$m_5$};
\node at (11.5,-1.5) {$m_6$};
\node at (18,-1.5) {$m_7$};
\node at (25.5,-1.5) {$m_1$};
\node at (1,-2.5) {$m_2$};
\node at (3.5,-2.5) {$m_{4'}$};
\node at (8,-2.5) {$m_5$};
\node at (15,-2.5) {$m_7$};
\node at (1.5,-3.5) {$m_{4'}$};
\node at (6,-3.5) {$m_5$};
\node at (13,-3.5) {$m_7$};
\node at (18,-3.5) {$m_2$};
\node at (3,-4.5) {$m_3$};
\node at (7.5,-4.5) {$m_{4'}$};
\node at (13,-4.5) {$m_7$};
\node at (4,-5.5) {$m_7$};
\node at (12.5,-5.5) {$m_3+m_{4'}$};
\node at (0.5,-6.5) {$m_4$};
\node at (4,-6.5) {$m_5$};
\node at (3.5,-7.5) {$m_4+m_5$};
\end{tikzpicture}
\]
Note that the $8$'s in the even rows can overlap any of the blocks depending on $m_1$, $m_2$, and $m_3+m_{4'}$ compares with $m_{4'}+m_5+m_6+m_7$, $m_{4'}+m_5+m_7$, and $m_7$ respectively.

\subsection{Diagram automorphism}

We can construct the type $E_6$ crystal decomposition as follows. Let $I_{0,7} := I_0 \setminus \{7\}$. We construct the $I_{0,7}$-highest weight elements by considering semistandard tableaux in $B(s\clfw_7)$ consisting of the elements of
\[
7 \xleftarrow{\hspace{15pt}} \bseven6 \xleftarrow{\hspace{15pt}} \bseven1 \xleftarrow{\hspace{15pt}} \bseven.
\]
The computation is similar to~\eqref{eq:comp_graph_B7}. Similarly, by reversing the arrows and adding loops at every vertex, we obtain the composition graph $G_7(\clfw_7)$. Thus, the $I_{0,7}$-highest weight elements in $B(s\clfw_7)$ are given by
\begin{equation}
\label{eq:07_hw_elts}
7^{\otimes m_7} \otimes \bseven6^{\otimes m_{\bseven6}} \otimes \bseven1^{\otimes m_{\bseven1}} \otimes \bseven^{\otimes m_{\bseven}}
\end{equation}
(recall that we naturally identify this with a semistandard tableau) with $m_7 + m_{\bseven6} + m_{\bseven1} + m_{\bseven} = s$.

\begin{prop}
\label{prop:E6_decomp}
Let $m_b$ be the number of occurrences of $b$ in a single row tableau $R \in B(s\clfw_7)$. Then $R$ is an $I_{0,7}$-highest weight element. Moreover, as $E_6$ crystals, we have
\[
B(s\clfw_7) \iso \bigoplus_{m_{\bseven6} + m_{\bseven1} + m_{\bseven} \leq s} B(m_{\bseven6} \clfw_6 + m_{\bseven1} \clfw_1).
\]
\end{prop}

Note that the decomposition into $E_6$ crystals from Proposition~\ref{prop:E6_decomp} is not multiplicity free; indeed, the multiplicity of $B(m_{\bseven6} \clfw_6 + m_{\bseven1} \clfw_1)$ is $s - m_{\bseven6} - m_{\bseven1}$ (the number of distinct values $m_{\bseven}$ can take). However, we can distinguish each of the $I_{0,7}$-highest weight elements by using the extra information (in addition to the weight) of $\inner{\wt(b)}{\alpha^{\vee}_7} = m_7 - m_{\bseven1} - m_{\bseven6} - m_{\bseven}$. Furthermore, by the level-zero condition, we have
$
\inner{\wt(b)}{\alpha^{\vee}_0} =
m_{\bseven} - m_7 - m_{\bseven6} - m_{\bseven1}.
$
By also using $m_7 + m_{\bseven6} + m_{\bseven1} + m_{\bseven} = s$, we thus obtain the following.

\begin{prop}
\label{prob:twisted_isomorphism}
The map $\Phi \colon B^{7,s} \to B^{7,s}$ given by
\[
\Phi\bigl( 7^{\otimes a} \otimes \bseven6^{\otimes b} \otimes \bseven1^{\otimes c} \otimes \bseven^{\otimes d} \bigr) = 7^{\otimes d} \otimes \bseven6^{\otimes c} \otimes \bseven1^{\otimes b} \otimes \bseven^{\otimes a}
\]
and extended as a twisted $I_{0,7}$-crystal morphism is a twisted $I_{0,7}$-crystal isomorphism. Moreover $\Phi$ is an twisted crystal involution that is induced from the order $2$ diagram automorphism of $E_7^{(1)}$.
\end{prop}
By restricting to $I_{0,2}$-highest weight elements, we have that $\Phi = \psi^{-1} \circ \sigma \circ \psi$, where $\psi$ is the isomorphism from Proposition~\ref{prop:B1s_A7_decomp}. This can be seen by equating the weight and $\langle \wt(b), \alpha_7^{\vee} \rangle$.
Thus, $\Phi$ when restricted to the $A_7$ crystal of $B^{7,s}$ is the order $2$ diagram automorphism of~$A_7$.

As a consequence, we have that $B^{7,s}$ satisfies~\cite[Assumption~1]{FSS.2007}, and so we have the following by~\cite[Thm.~4.7]{FSS.2007}. (We refer the reader to~\cite{FSS.2007} for the precise definitions.)

\begin{cor}
The tensor product $B = (B^{7,s})^{\otimes m} \otimes \{u_{s\omega_0}\}$ is isomorphic to a Demazure subcrystal $B_w(s\omega_{\tau(0)})$, where $t_{\mu} = w \tau$ in the extended affine Weyl group with $\mu = -m \omega_7$.
Moreover, $B$ and $(B^{7,s})^{\otimes m}$ are connected.
\end{cor}

\subsection{Combinatorial \texorpdfstring{$R$}{R}-matrix}

We give an explicit description of the combinatorial $R$-matrix $R \colon B^{7,s} \otimes B^{7,s'} \to B^{7,s'} \otimes B^{7,s}$ by noting the classical decomposition of $B^{7,s} \otimes B^{7,s'}$ is multiplicity free as $E_7$ crystals.
We may assume $s \leq s'$ without loss of generality as the combinatorial $R$-matrix is an involution.
Thus it is sufficient to consider the $I_{0,7}$-highest weight elements of $B^{7,s}$, which is given by~\eqref{eq:07_hw_elts}.
Therefore, the $I_0$-highest weight elements have weight
\[
(m_{\bseven} - m_7 - m_{\bseven6} - m_{\bseven1} - s') \omega_0 + m_{\bseven1} \omega_1 + m_{\bseven6} \omega_6 + (s' + m_7 - m_{\bseven6} - m_{\bseven1} - m_{\bseven}) \omega_7.
\]
The fact that the decomposition into $E_7$ crystals is multiplicity free is exactly the same reasoning as in the proof of Proposition~\ref{prob:twisted_isomorphism}.
Since the combinatorial $R$-matrix must map classical components to classical components, we have the following.

\begin{thm}
For $s \leq s'$, the combinatorial $R$-matrix $R \colon B^{7,s} \otimes B^{7,s'} \mapsto B^{7,s'} \otimes B^{7,s}$ is defined by
\[
(7^{\otimes m_7} \otimes \bseven6^{\otimes m_{\bseven6}} \otimes \bseven1^{\otimes m_{\bseven1}} \otimes \bseven^{\otimes m_{\bseven}}) \otimes 7^{\otimes s'}
\mapsto
(7^{\otimes (s' - s) + m_7} \otimes \bseven6^{\otimes m_{\bseven6}} \otimes \bseven1^{\otimes m_{\bseven1}} \otimes \bseven^{\otimes m_{\bseven}}) \otimes 7^{\otimes s}
\]
and extended as an $I_0$-crystal isomorphism.
\end{thm}

\section{Conjectures for \texorpdfstring{$B^{1,s}$}{B1s}}
\label{sec:conjectures}

We conclude with a conjectural decomposition of $B^{1,s}$ into $A_7$ crystals. Recall that $A_7$ has a natural diagram symmetry $\sigma$ that respects the $E_7^{(1)}$ diagram symmetry. Thus, proving this conjecture using $E_7$ crystal decomposition of $B^{1,s} \iso \bigoplus_{k=0}^s B(k\clfw_1)$ could possibly lead to a proof of~\cite[Conj.~3.26]{JS10}.

\begin{conj}
\label{conj:A7_decomp_B1s}
Let $a,b,c,d \in \ZZ_{\geq 0}$ such that $a + 2b + 3c + d \leq s$. Then we have
\[
B^{1,s} \iso \bigoplus B(a(\fwA_1+\fwA_7) + b(\fwA_2+\fwA_6) + c(\fwA_3+\fwA_5) + d\fwA_4)^{\oplus m_{a,b,c,d}},
\]
as $A_7$ crystals, where the multiplicities are $m_{a,b,c,d} = m_{d,s-a-2b-3c}$, where
\[
m_{d,s'} = \sum_{i=M}^{d+1} \left\lceil \frac{i}{2} \right\rceil
\]
with $M =\max(d+1-(s'-d), 0)$.
\end{conj}

We compute the multiplicity of $\fwA_4$ in $B^{1,s}$ using \textsc{SageMath}~\cite{sage}:

\begin{lstlisting}
sage: for s in range(10):
....:     [sum(ceil(i/2) for i in range(max(0,2*d+1-s),d+1+1))
....:      for d in range(s+1)]
[1]
[1, 1]
[1, 2, 2]
[1, 2, 3, 2]
[1, 2, 4, 4, 3]
[1, 2, 4, 5, 5, 3]
[1, 2, 4, 6, 7, 6, 4]
[1, 2, 4, 6, 8, 8, 7, 4]
[1, 2, 4, 6, 9, 10, 10, 8, 5]
[1, 2, 4, 6, 9, 11, 12, 11, 9, 5]
\end{lstlisting}

We compute the decomposition of $B^{7,s}$ into $A_7$ crystals using \textsc{SageMath}:

\begin{lstlisting}
sage: def compute_branching(s):
....:     A7 = WeylCharacterRing(['A',7], style="coroots")
....:     E7 = WeylCharacterRing(['E',7], style="coroots")
....:     La = E7.fundamental_weights()
....:     chi = sum(E7(k*La[1]) for k in range(s+1))
....:     return chi.branch(A7, rule="extended")
sage: compute_branching(1)
A7(0,0,0,0,0,0,0) + A7(0,0,0,1,0,0,0) + A7(1,0,0,0,0,0,1)
sage: compute_branching(2)
2*A7(0,0,0,0,0,0,0) + 2*A7(0,0,0,1,0,0,0) + A7(0,0,0,2,0,0,0)
 + A7(0,1,0,0,0,1,0) + A7(1,0,0,0,0,0,1) + A7(1,0,0,1,0,0,1)
 + A7(2,0,0,0,0,0,2)
sage: compute_branching(3)
2*A7(0,0,0,0,0,0,0) + 3*A7(0,0,0,1,0,0,0) + 2*A7(0,0,0,2,0,0,0)
 + A7(0,0,1,0,1,0,0) + A7(0,1,0,0,0,1,0) + 2*A7(1,0,0,0,0,0,1)
 + 2*A7(1,0,0,1,0,0,1) + A7(0,1,0,1,0,1,0) + A7(0,0,0,3,0,0,0)
 + A7(1,0,0,2,0,0,1) + A7(1,1,0,0,0,1,1) + A7(2,0,0,0,0,0,2)
 + A7(2,0,0,1,0,0,2) + A7(3,0,0,0,0,0,3)
\end{lstlisting}

We have verified Conjecture~\ref{conj:A7_decomp_B1s} for $s \leq 9$ by using a heavily optimized version of the branching rule code in \textsc{SageMath}~\cite{sage}.

One way to construct the $A_7$ highest weight elements would be to use the $A_6$ highest weight elements, which we can compute from the composition graph in Figure~\ref{fig:I2_comp_graph} (equivalently Figure~\ref{fig:compact_I2}). From there, we will want the elements invariant (as a set) under the diagram automorphism from~\cite{JS10} as highest weight elements of weight $\eta$ must map to a highest weight element of weight $-w_0 \eta$ under the $A_7$ diagram automorphism $\sigma$.

As a step towards proving Conjecture~\ref{conj:A7_decomp_B1s}, we show the analog of~\cite[Prop.~2.13]{JS10} that characterizes the elements in $B(2\clfw_1) \subseteq B(\clfw_1)^{\otimes 2}$.

\begin{prop}
We have
\[
(b_1 \otimes c_1) \otimes (b_2 \otimes c_2) \in B(2\clfw_1) \subseteq B(\clfw_1)^{\otimes 2} \subseteq B(\clfw_7)^{\otimes 4}
\]
if and only if
  $b_1 \leq b_2$ and $c_1 \leq c_2$ (the comparisons are in $B(\clfw_7)$) and
  $(b_1 \otimes c_1) < (b_2 \otimes c_2)$ (the comparison is in $B(\clfw_1)$).
\end{prop}

\begin{proof}
This is a finite computation that can be done, \textit{e.g.}, by using the following \textsc{SageMath}~\cite{sage} code:
\begin{lstlisting}
sage: L = crystals.Letters(['E',7])
sage: x = L.highest_weight_vector().f_string([7,6,5,4,2,3,4,5,6,7])
sage: A = tensor([x,L.highest_weight_vector()]).subcrystal()
sage: S = tensor([A.module_generators[0].value,
....:             A.module_generators[0].value]).subcrystal()
sage: all(P.le(x.value[0][1], x.value[1][1]) for x in S) # Check the b_i condition
True
sage: all(P.le(x.value[0][0], x.value[1][0]) for x in S) # Check the c_i condition
True
sage: P = Poset(L.digraph())
sage: PA = Poset(A.digraph())
sage: data = [[x.value[0], x.value[1]] for x in S]
sage: T = tensor([A, A])
sage: all((not PA.le(A(x[0].value), A(x[1].value))) or x[0].value == x[1].value
....:     for x in T if P.le(x[0].value[0], x[1].value[0])
....:               and P.le(x[0].value[1], x[1].value[1])
....:               and [x[0].value, x[1].value] not in data)
True
\end{lstlisting}
\end{proof}

\appendix
\section{\textsc{SageMath} code for composition graphs}

To compute the composition graph in Figure~\ref{fig:I2_comp_graph}, we first do some setup:

\begin{lstlisting}
sage: L = crystals.Letters(['E',7])
sage: x = L.highest_weight_vector().f_string([7,6,5,4,2,3,4,5,6,7])
sage: A = tensor([x,L.highest_weight_vector()]).subcrystal()
sage: TA = tensor([A.module_generators[0], A.module_generators[0]]).subcrystal()
sage: _ = A.list()
\end{lstlisting}

To compute the composition graph $G_2(\clfw_1)$, we run the following functions:

\begin{lstlisting}
def check_le(x, y):
    return tensor([x,y]) in TA

def composition_graph(J):
    I = A.index_set()
    ImJ = sorted(set(I) - set(J))
    G = DiGraph([[b for b in A if b.is_highest_weight(ImJ)], check_le])
    num_verts = 0
    while num_verts != G.num_verts():
        num_verts = G.num_verts()
        verts = set(G.vertices())
        for b in A:
            ep = set([i for i in I if b.epsilon(i) > 0])
            Jplus = set(list(J) + [i for i in ImJ
                                   if any(bp.phi(i) > 0 and check_le(b,bp)
                                          for bp in verts)
                                  ]
                       )
            if ep.issubset(Jplus):
                G.add_vertex(b)
                for bp in verts:
                    if check_le(b, bp):
                        G.add_edge(b, bp)
    loops = G.loops()
    G = G.transitive_reduction()
    for l in loops:
        G.add_edge(*l)
    G.set_latex_options(format='dot2tex')
    return G
\end{lstlisting}

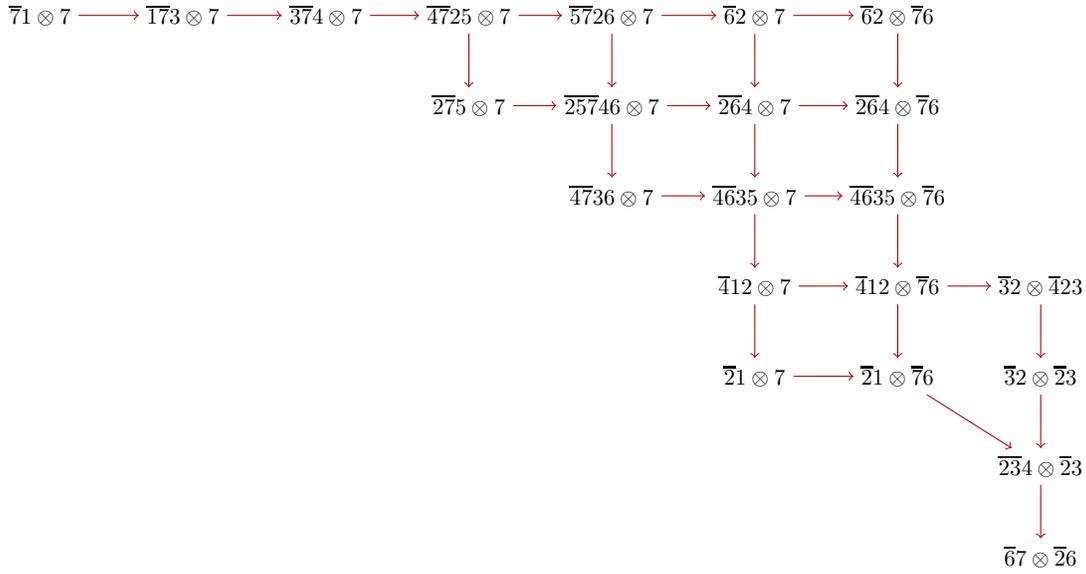
\begin{figure}
\[
\begin{tikzpicture}[xscale=1.9,yscale=1.2,every node/.style={scale=0.85}]
\node (b71t7) at (0,0) {$\bseven1 \otimes 7$};
\node (b1b73t7) at (1,0) {$\bon\bseven3 \otimes 7$};
\node (b3b74t7) at (2,0) {$\bth\bseven4 \otimes 7$};
\node (b4b725t7) at (3,0) {$\bfo\bseven25 \otimes 7$};
\node (b5b726t7) at (4,0) {$\bfive\bseven26 \otimes 7$};
\node (b62t7) at (5,0) {$\bsix2 \otimes 7$};
\node (b62tb76) at (6,0) {$\bsix2 \otimes \bseven6$};
\node (b2b75t7) at (3,-1) {$\btw\bseven5 \otimes 7$};
\node (b2b5b746t7) at (4,-1) {$\btw\bfive\bseven46 \otimes 7$};
\node (b2b64t7) at (5,-1) {$\btw\bsix4 \otimes 7$};
\node (b2b64tb76) at (6,-1) {$\btw\bsix4 \otimes \bseven6$};
\node (b4b736t7) at (4,-2) {$\bfo\bseven36 \otimes 7$};
\node (b4b635t7) at (5,-2) {$\bfo\bsix35 \otimes 7$};
\node (b4b635tb76) at (6,-2) {$\bfo\bsix35 \otimes \bseven6$};
\node (b412t7) at (5,-3) {$\bfo12 \otimes 7$};
\node (b412tb76) at (6,-3) {$\bfo12 \otimes \bseven6$};
\node (b32tb423) at (7,-3) {$\bth2 \otimes \bfo23$};
\node (b21t7) at (5,-4) {$\btw1 \otimes 7$};
\node (b21tb76) at (6,-4) {$\btw1 \otimes \bseven6$};
\node (b32tb23) at (7,-4) {$\bth2 \otimes \btw3$};
\node (b2b34tb23) at (7,-5) {$\btw\bth4 \otimes \btw3$};
\node (b67tb26) at (7,-6) {$\bsix7 \otimes \btw6$};
\draw[->,darkred] (b71t7) -- (b1b73t7);
\draw[->,darkred] (b1b73t7) -- (b3b74t7);
\draw[->,darkred] (b3b74t7) -- (b4b725t7);
\draw[->,darkred] (b4b725t7) -- (b5b726t7);
\draw[->,darkred] (b5b726t7) -- (b62t7);
\draw[->,darkred] (b62t7) -- (b62tb76);
\draw[->,darkred] (b4b725t7) -- (b2b75t7);
\draw[->,darkred] (b5b726t7) -- (b2b5b746t7);
\draw[->,darkred] (b62t7) -- (b2b64t7);
\draw[->,darkred] (b62tb76) -- (b2b64tb76);
\draw[->,darkred] (b2b75t7) -- (b2b5b746t7);
\draw[->,darkred] (b2b5b746t7) -- (b2b64t7);
\draw[->,darkred] (b2b64t7) -- (b2b64tb76);
\draw[->,darkred] (b2b5b746t7) -- (b4b736t7);
\draw[->,darkred] (b2b64t7) -- (b4b635t7);
\draw[->,darkred] (b2b64tb76) -- (b4b635tb76);
\draw[->,darkred] (b4b736t7) -- (b4b635t7);
\draw[->,darkred] (b4b635t7) -- (b4b635tb76);
\draw[->,darkred] (b4b635t7) -- (b412t7);
\draw[->,darkred] (b4b635tb76) -- (b412tb76);
\draw[->,darkred] (b412t7) -- (b412tb76);
\draw[->,darkred] (b412tb76) -- (b32tb423);
\draw[->,darkred] (b412t7) -- (b21t7);
\draw[->,darkred] (b412tb76) -- (b21tb76);
\draw[->,darkred] (b32tb423) -- (b32tb23);
\draw[->,darkred] (b32tb23) -- (b2b34tb23);
\draw[->,darkred] (b21t7) -- (b21tb76);
\draw[->,darkred] (b21tb76) -- (b2b34tb23);
\draw[->,darkred] (b2b34tb23) -- (b67tb26);
\end{tikzpicture}
\]
\caption{The composition graph $G_2(\clfw_1)$ to compute the $(I_0 \setminus \{2\})$-highest weight elements. We have suppressed the loops that occur at every node except $\bth2 \otimes \btw3$.}
\label{fig:I2_comp_graph}
\end{figure}

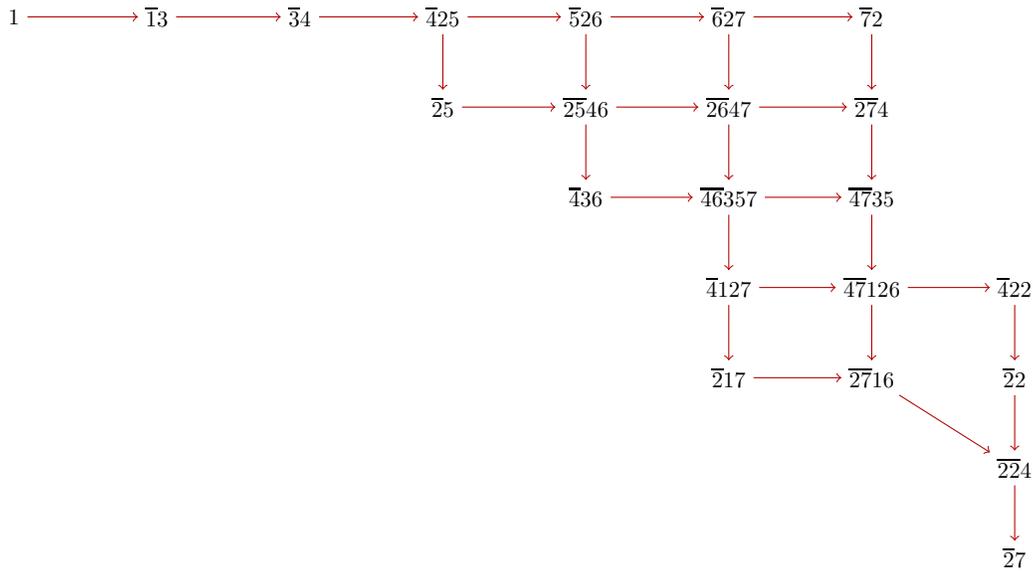
\begin{figure}
\[
\begin{tikzpicture}[xscale=1.9,yscale=1.2,every node/.style={scale=0.85}]
\node (b71t7) at (0,0) {$1$};
\node (b1b73t7) at (1,0) {$\bon3$};
\node (b3b74t7) at (2,0) {$\bth4$};
\node (b4b725t7) at (3,0) {$\bfo25$};
\node (b5b726t7) at (4,0) {$\bfive26$};
\node (b62t7) at (5,0) {$\bsix27$};
\node (b62tb76) at (6,0) {$\bseven2$};
\node (b2b75t7) at (3,-1) {$\btw5$};
\node (b2b5b746t7) at (4,-1) {$\btw\bfive46$};
\node (b2b64t7) at (5,-1) {$\btw\bsix47$};
\node (b2b64tb76) at (6,-1) {$\btw\bseven4$};
\node (b4b736t7) at (4,-2) {$\bfo36$};
\node (b4b635t7) at (5,-2) {$\bfo\bsix357$};
\node (b4b635tb76) at (6,-2) {$\bfo\bseven35$};
\node (b412t7) at (5,-3) {$\bfo127$};
\node (b412tb76) at (6,-3) {$\bfo\bseven126$};
\node (b32tb423) at (7,-3) {$\bfo22$};
\node (b21t7) at (5,-4) {$\btw17$};
\node (b21tb76) at (6,-4) {$\btw\bseven16$};
\node (b32tb23) at (7,-4) {$\btw2$};
\node (b2b34tb23) at (7,-5) {$\btw\btw4$};
\node (b67tb26) at (7,-6) {$\btw7$};
\draw[->,darkred] (b71t7) -- (b1b73t7);
\draw[->,darkred] (b1b73t7) -- (b3b74t7);
\draw[->,darkred] (b3b74t7) -- (b4b725t7);
\draw[->,darkred] (b4b725t7) -- (b5b726t7);
\draw[->,darkred] (b5b726t7) -- (b62t7);
\draw[->,darkred] (b62t7) -- (b62tb76);
\draw[->,darkred] (b4b725t7) -- (b2b75t7);
\draw[->,darkred] (b5b726t7) -- (b2b5b746t7);
\draw[->,darkred] (b62t7) -- (b2b64t7);
\draw[->,darkred] (b62tb76) -- (b2b64tb76);
\draw[->,darkred] (b2b75t7) -- (b2b5b746t7);
\draw[->,darkred] (b2b5b746t7) -- (b2b64t7);
\draw[->,darkred] (b2b64t7) -- (b2b64tb76);
\draw[->,darkred] (b2b5b746t7) -- (b4b736t7);
\draw[->,darkred] (b2b64t7) -- (b4b635t7);
\draw[->,darkred] (b2b64tb76) -- (b4b635tb76);
\draw[->,darkred] (b4b736t7) -- (b4b635t7);
\draw[->,darkred] (b4b635t7) -- (b4b635tb76);
\draw[->,darkred] (b4b635t7) -- (b412t7);
\draw[->,darkred] (b4b635tb76) -- (b412tb76);
\draw[->,darkred] (b412t7) -- (b412tb76);
\draw[->,darkred] (b412tb76) -- (b32tb423);
\draw[->,darkred] (b412t7) -- (b21t7);
\draw[->,darkred] (b412tb76) -- (b21tb76);
\draw[->,darkred] (b32tb423) -- (b32tb23);
\draw[->,darkred] (b32tb23) -- (b2b34tb23);
\draw[->,darkred] (b21t7) -- (b21tb76);
\draw[->,darkred] (b21tb76) -- (b2b34tb23);
\draw[->,darkred] (b2b34tb23) -- (b67tb26);
\end{tikzpicture}
\]
\caption{The composition graphs of Figure~\ref{fig:I2_comp_graph} with every node written in ``compact form,'' where a $k$ adds $1$ to $\varphi_k(b)$ and $\overline{k}$ adds $1$ to $\varepsilon_k(b)$. Recall that the only vertex that does not have a loop is $\btw2$.}
\label{fig:compact_I2}
\end{figure}

\bibliographystyle{plain}
\bibliography{kr_crystals}{}
\end{document}